\documentclass[preprint,12pt]{amsart}

\usepackage[margin=1.5in]{geometry}

\def\Z{{\mathbb{Z}}}
\def\Q{{\mathbb{Q}}}

\def\C{{\mathbb{C}}}
\def\P{{\mathbb{P}}}
\usepackage{amsmath}
\usepackage{amsfonts}
\usepackage{amssymb}
\usepackage{xypic}
\usepackage{amsthm}
\usepackage{url}

\theoremstyle{plain}
\newtheorem{theo}{Theorem}[section]
\newtheorem{coro}[theo]{Corollary}
\newtheorem{lemma}[theo]{Lemma}
\newtheorem{prop}[theo]{Proposition}
\theoremstyle{definition}
\newtheorem{defi}[theo]{Definition}
\newtheorem{ex}[theo]{Example}
\newtheorem{rem}[theo]{Remark}

\makeatletter
\def\ps@pprintTitle{%
  \let\@oddhead\@empty
  \let\@evenhead\@empty
  \let\@oddfoot\@empty
  \let\@evenfoot\@oddfoot
}
\makeatother

\title{Elliptic curves on abelian varieties}

\author{Robert Auffarth II}

\address{R. Auffarth II \\Departamento de Matem\'aticas, Facultad de
Ciencias, Universidad de Chile, Santiago\\Chile}
\email{rfauffar@mat.puc.cl }

\thanks{Partially supported by Fondecyt Grant 3150171}

\begin{document}

\maketitle

\begin{abstract}

Given a principally polarized abelian variety $(A,\Theta)$, we give a characterization of all elliptic curves that lie on $A$ in terms of intersection numbers of divisor classes in its N\'eron-Severi group.  
\end{abstract}

\section{Introduction}

The problem of finding elliptic curves on abelian varieties has a long history, originating with works of Abel and Jacobi on the decomposition of abelian integrals. A very general abelian variety contains no elliptic curve, but abelian varieties that do appear frequently in examples.

In their paper \cite{ES}, Ekedahl and Serre found examples of Jacobian varieties that split isogenously as the product of elliptic curves for certain curves up to genus 1297. A question that they pose is: For what numbers $g$ does there exist a (smooth projective) curve of genus $g$ whose Jacobian splits isogenously (or isomorphically) as the product of elliptic curves? Another interesting question is: For a fixed genus $g$, what is the maximum amount of elliptic curves that appear in any isogeny decomposition of a Jacobian variety of dimension $g$?

Most examples of abelian varieties that contain elliptic curves have been found using techniques such as group actions, but no comprehensive theory has been established in general. In dimension 2, Humbert \cite{Humbert} gave a description of all 2-dimensional principally polarized abelian varieties that are non-simple (that is, that contain a non-trivial abelian subvariety) in terms of their period matrices. The case of an irreducible principal polarization has recently been addressed in dimension 2 by Kani \cite{KaniJ} in a more algebraic setting.

In this paper, we address the issue of when an abelian variety defined over an algebraically closed field contains an elliptic curve, and we give a characterization of all elliptic curves on an abelian variety in terms of numerical divisor classes. This seems to be a first step in being able to answer the questions posed above. Kani \cite{Kani} described all abelian surfaces that contain an elliptic curve by means of intersection theory, and in this paper we generalize his methods to arbitrary dimension.

Throughout the paper, we let $A$ be an abelian variety of dimension $n$ defined over an algebraically closed field, and we fix once and for all an ample divisor $\Theta$ on $A$. An \emph{abelian divisor} will be an abelian subvariety of $A$ of codimension 1, seen as a prime (Weil) divisor on $A$.

One can prove that if $A$ contains an abelian subvariety $W$, then $W$ has an \emph{abelian complement}; that is, there exists an abelian subvariety $Y$ in $A$ such that the addition map $W\times Y\to A$ is an isogeny. There is a canonical way to find the abelian complement (\cite{BL2}), and so there is a bijection between abelian subvarieties of $A$ of dimension $r$ and abelian subvarieties of codimension $r$. We will always talk about the abelian complement of an abelian subvariety in the sense of Birkenhake and Lange \cite{BL2}.

We are interested in characterizing abelian varieties that contain an elliptic curve, and thus an abelian divisor, via intersection theory. Let $\frak{A}^*(A):=\bigoplus_{i=0}^n\frak{A}^i(A)$ be the \emph{Chow ring of $A$ modulo algebraic equivalence}, where $\frak{A}^i(A)$ denotes the group of algebraic cycles of codimension $i$ on $A$ modulo algebraic equivalence. We take algebraic equivalence so that we may translate a cycle by an element of $A$ and not affect its algebraic class. For $i=1$, $\frak{A}^1(A)$ is the \emph{N\'eron-Severi group} of $A$, and will be denoted by $\mbox{NS}(A)$. This is the group we will be concentrating on, and we will abuse notation throughout the paper by intersecting numerical (which in this case is the same as algebraic) classes of divisors and divisors interchangeably. If $D$ is a divisor, then $[D]$ will denote the algebraic equivalence class of $D$.

If $G$ is a free abelian group, then we will say that $g\in G$ is a \emph{primitive} element of $G$ if $G/\langle g\rangle$ is torsion-free.

Our first main result is the following:

\begin{theo}
Let $A$ be an abelian variety of dimension $n$, and let $\Theta$ be a fixed ample divisor on $A$. Then the map $Z\mapsto [Z]$ induces a bijective correspondence between abelian divisors on $A$ and primitive elements $\alpha\in \mbox{NS}(A)$ that satisfy $\alpha^2=0$ in $\frak{A}^*(A)$ and $(\alpha\cdot\Theta^{n-1})>0$. In particular, $A$ contains an elliptic curve if and only if there exists a non-zero class $\alpha\in \mbox{NS}(A)$ that satisfies $\alpha^2=0$.
\end{theo}

Let $\natural:\mbox{NS}(A)\to \mbox{NS}(A)$ denote the endomorphism
$$\alpha\mapsto\alpha^\natural=(\Theta^n)\alpha-(\alpha\cdot\Theta^{n-1})[\Theta];$$
it is easily seen to be a sort of projection away from $\Z[\Theta]$ with respect to the pairing $(\alpha,\beta)\mapsto(\alpha\cdot\beta\cdot\Theta^{n-2})$.
 We define the homogeneous polynomials
$$q_r(\alpha):=-\frac{1}{(r-1)(\Theta^n)}((\alpha^\natural)^r\cdot\Theta^{n-r})$$
for $2\leq r\leq n$. For $n=2$, we get precisely the quadratic form that Kani introduces in his paper \cite{Kani}. We can see these as forms on the \emph{polarized N\'eron-Severi group} $\mbox{NS}(A,\Theta):=\mbox{NS}(A)/\Z[\Theta]$.

Recall that $(A,\Theta)$ is a \emph{principally polarized abelian variety} (ppav) if the isogeny $A\to A^\vee$ (where $A^\vee=\mbox{Pic}^0(A)$ is the dual abelian variety of $A$) induced by $\Theta$ is an isomorphism (or equivalently, $h^0(A,\mathcal{O}_A(\Theta))=1$).

Our second result states:

\begin{theo}
Let $(A,\Theta)$ be a ppav. Then there is a bijective correspondence between abelian divisors $Z\subseteq A$ with $(Z\cdot\Theta^{n-1})=d$ and primitive numerical classes $\alpha\in \mbox{NS}(A,\Theta)$ that satisfy $q_r(\alpha)=(-1)^rd^r$ for $r=2,\ldots,n$, given by $Z\mapsto[Z]$.
\end{theo}

This theorem can be used to produce Humbert-style equations for the moduli space of complex ppavs that contain an abelian divisor (and thus elliptic curves) of a certain degree. \\

\noindent\emph{Acknowledgements:} The work presented here is part of my Ph.D. thesis, and I am very grateful to my advisor Rub\'{\i} Rodr\'{\i}guez for her support and mathematical insight. I would like to thank Igor Dolgachev for the invitation to study at the University of Michigan for a semester, for very helpful discussions and for asking hard questions that contributed to this paper. I am also grateful to Samuel Grushevsky and Giancarlo Urz\'{u}a for their mathematical advice and for giving me interesting ideas on where to direct the techniques developed here.

\section{Abelian divisors}

Our first goal is to be able to characterize abelian divisors by their numerical classes. We first characterize them among all prime divisors.

\begin{prop}\label{1}
Let $Z$ be a prime divisor on an abelian variety $A$ of dimension $n$. Then $Z$ is the translation of an abelian divisor if and only if $[Z]^2=0$ in $\frak{A}^*(A)$.
\end{prop}
\begin{proof}
 If $Z$ is the translation of an abelian divisor and $z\in Z$, let $x\notin Z-z:=t_{-z}(Z)$, where $t_a:A\to A$ for $a\in A$ is translation by $a$. Then $Z\cap (Z+x)=\varnothing$, and so in particular $[Z]^2=0$.
Conversely, assume that $Z$ is a prime divisor and $(Z^2\cdot \Theta^{n-2})=0$. By translating, we can assume that $Z$ contains $0$. We see that if $x\in Z$, then $0\in Z\cap(Z-x)$, and therefore, if $Z\cap(Z-x)\neq Z$, this would be a subvariety of $A$ of codimension 2 and hence $(Z^2\cdot\Theta^{n-2})>0$, a contradiction. This implies that $Z=Z-x$. Similarly, we see that for $x,y\in Z$, $Z-(x+y)=(Z-x)-y=Z-y=Z$, and we therefore conclude that $Z$ is a group. Since we can see $Z$ as an irreducible subvariety of $A$, we obtain that it is an abelian subvariety of codimension 1.
\end{proof}

We see that abelian divisors correspond to certain elements $\alpha\in\mbox{NS}(A)$ such that $\alpha^2=0$. The question we would like to answer is: How can we characterize abelian divisors among all such elements?

\begin{prop}\label{3}
If $Z$ is an effective divisor, then $Z\equiv mY$ for some abelian divisor if and only if $[Z]^2=0$ in $\frak{A}^*(A)$ (and this occurs if and only if $(Z^2\cdot\Theta^{n-2})=0$).
\end{prop}

\begin{proof} Assume that $[Z]^2=0$ and $Z\neq 0$. Write $Z=\sum m_iF_i$, where the $F_i$ are irreducible codimension 1 subvarieties of $A$ and $m_i>0$. We have that
$$0=(Z^2\cdot \Theta^{n-2})=\sum_{i,j}m_im_j(F_i\cdot F_j\cdot \Theta^{n-2})\geq0.$$
In particular, $(F_i^2\cdot \Theta^{n-2})=0$, and so by Proposition~\ref{1} we have that each $F_i$ is the translate of an abelian divisor. Since we can move all the $F_i$ inside their numerical equivalence classes, assume that all are abelian subvarieties. If $F_i$ and $F_j$ are different, for instance, we have that $(F_i\cap F_j)_0$ is an abelian subvariety of codimension 2. However, this contradicts the fact that $(F_i\cdot F_j\cdot\Theta^{n-2})=0$. Therefore we must have that all the $F_i$ are the same, and the result follows.
\end{proof}

\begin{defi}
We will say that a class $\alpha\in \mbox{NS}(A)$ is \emph{primitive} if every time we have $\alpha=m\beta$ for some $\beta\in \mbox{NS}(A)$, then $m=\pm1$. This is equivalent to $\mbox{NS}(A)/\Z\alpha$ being torsion free. We will say that a class $\alpha\in \mbox{NS}(A)$ is \emph{effective} if there exists an effective divisor $D$ on $A$ such that $[D]=\alpha$.
\end{defi}

The following two results are fundamental to all the analysis that follows. The results are well-known (see Birkenhake-Lange \cite{BL} and Bauer \cite{Bauer} for simple proofs over $\C$), but for lack of an adequate reference over a general algebraically closed field, we include the proofs here.

\begin{theo}[Nakai-Moishezon Criterion]
 Let $A$ be an abelian variety of dimension $n$ and let $\Theta$ be an ample divisor. Then a divisor $D$ on $A$ is ample if and only if $(D^i\cdot\Theta^{n-i})>0$ for all $i\leq n$.
\end{theo}

\begin{proof}
The only interesting part is to prove that if $(D^i\cdot\Theta^{n-i})>0$ for all $i$, then $D$ is ample. We see that if $D$ satisfies this, then in particular $(D^n)>0$. By Mumford \cite{Mum}, p. 145, we have that $H^{p}(A,D)=0$ for all $p\neq i(D)$ and $H^{i(D)}(A,D)\neq0$, where $i(D)$ is the number of positive roots of the polynomial $P(t)$ defined by $P(m):=\chi(m\Theta+D)$. By Riemann-Roch,
$$P(m)=\frac{((m\Theta+D)^n)}{n!},$$
and by our assumption on $D$, all the coefficients of this polynomial are positive. Therefore $i(D)=0$, and so $H^0(A,D)\neq0$. This means that $D$ is linearly equivalent to an effective divisor, and by Application 1 on page 57 of Mumford \cite{Mum}, we get that $D$ is ample.
\end{proof}

\begin{prop}\label{4}
If $A$ is an abelian variety of dimension $n$ and $D$ is a divisor on $A$, then $D$ is numerically equivalent to an effective divisor if and only if $(D^i\cdot\Theta^{n-i})\geq0$ for $1\leq i\leq n$.
\end{prop}

\begin{proof}
If $D$ is numerically equivalent to an effective divisor, then clearly $(D^i\cdot\Theta^{n-i})\geq0$ for $1\leq i\leq n$.
For the other direction, assume that this inequality holds for $1\leq i\leq n$. Using the Nakai-Moishezon Criterion with $D+m\Theta$ and $\Theta$, we first observe that $D+m\Theta$ is ample for all $m\geq0$, and so $D$ is nef.

We will now proceed by induction on $n$ to prove the proposition. We see that for $n=1$ the proof is trivial. We then assume that $n>1$. If $(D^n)>0$, then for the same reasons as in the previous proof we have that $H^0(A,D)\neq0$, and so $D$ is linearly equivalent to an effective divisor.

If $(D^n)=0$, then $K(D):=\{x\in A:t_x^*D\sim D\}$ is not finite and there exists a divisor $D'$ on $A/K(D)_0$ such that $D-\pi^*D'$ is numerically trivial, where $\pi:A\to A/K(D)_0$ is the natural projection and $K(D)_0$ denotes the connected component of $0$ in $K(D)$. Since $\pi$ is proper and surjective and $\pi^*D'$ is nef, we also have that $D'$ is nef (this can be shown using the projection formula). By our induction hypothesis, we have that $D'$ is numerically equivalent to an effective divisor, and therefore $\pi^*D'\equiv D$ is numerically equivalent to an effective divisor.
\end{proof}

This proposition actually shows that if $\alpha^2=0$ in $\frak{A}^*(A)$, then either $\alpha$ or $-\alpha$ comes from an effective divisor.

\begin{coro}\label{5}
If $\alpha\in \mbox{NS}(A)$, then $\alpha=m[Z]$ for some abelian divisor $Z$ and some $m\in\Z$ if and only if $\alpha^2=0$ in $\frak{A}^*(A)$.
\end{coro}

Using Proposition~\ref{4}, we can prove the following:

\begin{lemma}\label{6}
The class of an abelian divisor is primitive.
\end{lemma}
\begin{proof}
Let $Z$ be an abelian divisor, and assume that $Z\equiv mD$ for some divisor $D$ such that $[D]$ is primitive. Suppose that $m>0$; if not, then we replace $D$ by $-D$. Since $Z$ is effective and $m>0$, by the previous proposition we can assume that $D$ is effective. Now using Proposition~\ref{3}, we see that $D\equiv Y$ for some abelian divisor $Y$, and so $Z\equiv mY$. However, we then see that $[Z\cdot Y]=0$ in $\frak{A}^*(A)$, and using the argument used in the proof of Proposition~\ref{3}, we get that $Y\equiv Z$. Therefore $(m-1)Z\equiv0$, and so $m=1$.
\end{proof}
\begin{rem}
Another way of proving this lemma in characteristic 0 is using the following criterion: The class of a divisor $D$ is primitive if and only if $A[m]\nsubseteq K(D)$ for some $m\in\Z$ (where $A[m]$ denotes the group of $m$-torsion points of $A$). If $D$ is an abelian divisor, then by cardinality $A[m]\nsubseteq K(D)=D$ for all $m$.
\end{rem}

Using this lemma and Proposition~\ref{3}, we get:

\begin{coro}\label{7}
A class $\alpha\in \mbox{NS}(A)$ comes from an abelian divisor if and only if it is effective, primitive and $\alpha^2=0$ in $\frak{A}^*(A)$.
\end{coro}

We define the \emph{degree} of a divisor $D$ on $A$ to be
$$\deg D:=(D\cdot\Theta^{n-1}).$$
In the same way we define the degree of an algebraic class. The degree of a curve $C$ on $A$ is defined analogously as $\deg C:=(C\cdot\Theta)$. Notice that if $\Theta$ is very ample, then the degree of a divisor coincides with the degree of $\varphi(D)$ in $\P^{h^0(\Theta)-1}$, where $\varphi:A\to\P^{h^0(\Theta)-1}$ is the embedding associated with $\Theta$.

With all we have said so far, we can prove our first main result:

\begin{theo}\label{8}
Let $A$ be an abelian variety of dimension $n$. Then the map $Z\mapsto [Z]$ induces a bijective correspondence between abelian divisors on $A$ and primitive elements $\alpha\in \mbox{NS}(A)$ that satisfy $\alpha^2=0$ in $\frak{A}^*(A)$ and $\deg\alpha>0$.
\end{theo}
\begin{proof} Injectivity was already proven. To show surjectivity, let $\alpha$ be a primitive class that satisfies $(\alpha\cdot\Theta^{n-1})>0$ and $\alpha^2=0$. Since $\alpha^2=0$, we have that $(\alpha^i\cdot\Theta^{n-i})=0$ for every $i\geq 2$. Proposition~\ref{4} says that $\alpha$ is numerically equivalent to (the class of) an effective divisor, and Corollary~\ref{7} says that $\alpha$ actually comes from an abelian divisor.
\end{proof}

\section{The forms $q_r$}

In this section we will mostly concentrate on principally polarized abelian varieties (ppavs); we will mention when this hypothesis is necessary.

Let $\natural:\mbox{NS}(A)\to \mbox{NS}(A)$ denote the endomorphism
$$\alpha\mapsto\alpha^\natural=(\Theta^n)\alpha-(\deg\alpha)[\Theta].$$
We define the homogeneous polynomials
$$q_r(\alpha):=-\frac{1}{(r-1)(\Theta^n)}((\alpha^\natural)^r\cdot\Theta^{n-r})$$
for $2\leq r\leq n$.

By expanding the right hand side, we have that
$$q_r(\alpha)=(-1)^{r}(\deg\alpha)^r+\frac{(\Theta^n)}{r-1}\sum_{m=2}^r{r \choose m}(\Theta^n)^{m-2}(-1)^{r-m+1}(\deg\alpha)^{r-m}
(\alpha^m\cdot\Theta^{n-m}).$$

\begin{lemma}\label{9}
If $[\Theta]$ is primitive in $\mbox{NS}(A)$ and $\alpha\in \mbox{NS}(A)$, then $q_r(\alpha)\leq 0$ for all $r=2,\ldots,n$ if and only if $\alpha\in\Z[\Theta]$.
\end{lemma}
\begin{proof}
It is clearly seen that $q_r([\Theta])=0$. Conversely, if $q_r(\alpha)\leq0$, then we would have that $((\alpha^\natural)^r\cdot \Theta^{n-r})\geq0$ for all $r$ (and by the definition of $\alpha^\natural$, $\deg\alpha^\natural=0$), so by Proposition~\ref{4} we have that $\alpha^\natural= m[D]$, where $D$ is an effective divisor on $A$ and $m\geq0$. But if $D\not\equiv 0$, then $\deg D>0$, a contradiction. Therefore $\alpha^\natural\equiv0$, and so $(\Theta^n)\alpha\in\Z[\Theta]$. Since $[\Theta]$ is primitive, we obtain that $\alpha\in\Z[\Theta]$.
\end{proof}

\begin{rem}\label{10}
We observe that if $\alpha\in \mbox{NS}(A)$ satisfies $\alpha^2=0$ in $\frak{A}^*(A)$, then
$$q_r(\alpha)=(-1)^{r}(\deg\alpha)^{r}$$
for all $r$.
\end{rem}

\begin{rem}\label{11}
It is easy to see that the forms $q_r$ descend to forms on $\mbox{NS}(A)/\Z[\Theta]$. In dimension 3, the previous lemma shows that $q_2$ is positive definite on $\mbox{NS}(A)/\mathbb{Z}[\Theta]$. Indeed, if $q_2(\alpha)\leq 0$ and $q_3(\alpha)\leq 0$, then Lemma~\ref{9} says that $\alpha\in\Z[\Theta]$. If $q_2(\alpha)\leq 0$ and $q_3(\alpha)\geq0$, then $q_2(-\alpha)\leq0$ and $q_3(-\alpha)\leq0$, and we have the same situation.
\end{rem}

\vspace{0.5cm}

\begin{lemma}\label{12}
If $Z$ is an abelian divisor on $A$ and $\Theta$ is a principal polarization, then $\deg Z=(n-1)!(E\cdot\Theta)$, where $E$ is the abelian complement of $Z$ in $A$.
\end{lemma}

\begin{proof}
Set $d:=(\Theta\cdot E)$; in other words, $\Theta$ restricted to $E$ is a divisor of degree $d$. Using Riemann-Roch, the fact that $K(\Theta|_Z)\simeq Z\cap E\simeq K(\Theta|_E)$ and $\chi(\Theta|_Z)^2=|K(\Theta|_Z)|$, we have that
$$\deg Z=(Z\cdot\Theta^{n-1})=((\Theta|_Z)^{n-1})=(n-1)!\chi(\Theta|_Z)=(n-1)!\chi(\Theta|_E)=(n-1)!d.$$
\end{proof}

The next three lemmas are technical in nature and will be used in the proof of our main theorem. The first of the three is elementary and well known.

\begin{lemma}\label{13}
If $m\in\Z\backslash\{\pm1\}$ and $n\in\Z_{>0}$, then $m^{n-1}\mid n!$ if and only if $m=\pm 2$ and $n$ is a power of $2$.
\end{lemma}
\begin{proof}
We will prove this for a prime number $p$ that divides $m$. If $p$ is a prime number such that $p^{n-1}\mid n!$, then Legendre's formula for the highest power of a prime appearing in $n!$ says that
$$n-1\leq\left\lfloor\frac{n}{p}\right\rfloor+\left\lfloor\frac{n}{p^2}\right\rfloor+\cdots+\left\lfloor\frac{n}{p^l}
\right\rfloor$$
for $l=\lfloor\log_p(n)\rfloor$. If $S$ denotes the right hand side of the inequality, we have that
$$S\leq n\left(\frac{1}{p}+\frac{1}{p^2}+\cdots+\frac{1}{p^l}\right)=n\left(\frac{1-\frac{1}{p^l}}{p-1}\right).$$
This is obviously less than or equal to $\frac{n}{p-1}$, and since $n-1\leq S$, we get that $p$ is necessarily 2. Replacing $p=2$ above and clearing the equations, we arrive at $n\leq2^l=2^{\lfloor\log_2(n)\rfloor}$. If $n$ is not a multiple of $2$, then this is impossible. Therefore we conclude that $n=2^k$ for some $k$ and $p=2$.
\end{proof}

\begin{defi}Let $\mbox{NS}(A,\Theta):=\mbox{NS}(A)/\Z[\Theta]$ be the \emph{polarized N\'eron-Severi group of} $A$. As we said above, it is easy to see that the forms $q_r$ are well defined on $\mbox{NS}(A,\Theta)$.
\end{defi}

\begin{lemma}\label{14}
Let $(A,\Theta)$ be a ppav. Then the class of an abelian divisor in $\mbox{NS}(A,\Theta)$ is primitive.
\end{lemma}
\begin{proof}
Let $Z$ be an abelian divisor on $A$, and assume that $Z\equiv mD+s\Theta$ for some divisor $D$ and $m,s\in\Z$. We can also assume that $s\neq 0$ and actually $(m,s)=1$, since $Z$ is primitive in $\mbox{NS}(A)$. Moreover, after changing $D$ with $-D$ if necessary, we can assume that $m>0$. Since $[Z]^2=0$ in $\frak{A}^*(A)$, we get the following formula:
\begin{equation}
\nonumber m^r(D^r\cdot\Theta^{n-r})=((Z-s\Theta)^r\cdot\Theta^{n-r})=(-s)^{r-1}(r\deg Z-s(\Theta^n))
\end{equation}
for $1\leq r\leq n$. Assume that $(D^n)\neq0$ and $(D^{n-1}\cdot\Theta)\neq0$. We see that
$$m^{n-1}(D^{n-1}\cdot\Theta)=(-s)^{n-2}((n-1)\deg Z-s(\Theta^n))$$
and
$$m^n(D^n)=(-s)^{n-2}(n\deg Z-s(\Theta^n)).$$
This means that $m^{n-1}\mid(n-1)\deg Z-s(\Theta^n)$ and $m^n\mid n\deg Z-s(\Theta^n)$, and so $m^{n-1}\mid\deg Z$. But then $m^{n-1}\mid s(\Theta^n)$ and so $m^{n-1}\mid n!$, since $(\Theta^n)=n!$ in this case. By Lemma~\ref{13} we conclude that $n=2^k$ for some $k$ and $m=2$. In this case, we have that $2^n\mid n!(d-s)$, where $\deg Z=d(n-1)!$. If $d$ is even, then $d-s$ is odd and so $2^n\mid n!$. Based on the previous lemma, it is easy to see that this is impossible. If $d$ is odd, then
$$2^{n-1}\mid(n-1)!((n-1)d-sn)=(2^k-1)!((2^k-1)d-s2^k),$$
and since $(2^k-1)d-s2^k$ is odd, we have that $2^{n-1}\mid (n-1)!$, a contradiction. Therefore we must have that $m=1$.

If $(D^n)=0$, then $\deg Z=(n-1)!s$ and $(D^{n-1}\cdot\Theta)\neq0$. Therefore, $m^{n-1}\mid(n-1)(n-1)!s-sn!$, and so $m^{n-1}\mid(n-1)!$. Thus $m=1$. If $(D^{n-1}\cdot\Theta)=0$, then $(n-1)\deg Z=sn!$. Let $\deg Z=d(n-1)!$. Then $m^n\mid n!(d-s)$, and since $d=\frac{sn}{n-1}$, we get that
$$m^n\mid n!\left(\frac{sn}{n-1}-s\right)=s(n^2(n-2)!-n!)=s(n-2)!n.$$
Therefore, if $p$ is a prime that divides $m$, we have that either $p^n\mid(n-2)!$ or $p^n\mid n$, which is impossible based on what we have said above. We conclude that $m=1$, and so $[Z]$ is primitive in $\mbox{NS}(A,\Theta)$.
\end{proof}

We observe that this is no longer true when $\Theta$ is not a principal polarization. For instance, take $A=E_1\times E_2$ and $\Theta=\{0\}\times E_2+2(E_1\times \{0\})$ where $E_1$ and $E_2$ are elliptic curves. Putting $Z=\{0\}\times E_2$, we get that $Z=-2(E_1\times\{0\})+\Theta$, and so is not primitive in $\mbox{NS}(A,\Theta)$. This gives a counterexample to Theorem 3.2 of Kani \cite{Kani} (which implies that the class of an elliptic curve in $\mbox{NS}(A,\Theta)$ is primitive for an arbitrary primitive polarization). 

\begin{lemma}\label{15}
Take $\alpha\in\mbox{NS}(A)$ and let $d,k\in\Z$ be such that $k$ is positive, $\deg\alpha=d$ and $q_r(\alpha)=(-1)^rk^r$ for $r=2,\ldots,n$. Then $d-k\equiv0\mbox{ (mod }n!)$.
\end{lemma}
\begin{proof}
Let $x_r:=n!^{r-1}(\alpha^r\cdot\Theta^{n-r})$. It is easy to see, by expanding the definition of $q_r$, that for $r\geq3$
$$x_r=(r-1)(-1)^r(d^r-k^r)+\sum_{m=2}^{r-1}\binom{r}{m}(-1)^{r-m+1}d^{r-m}x_m$$
and $x_2=d^2-k^2$. By using induction, we arrive at the following expression for the general term:
$$x_r=(d-k)^{r-1}(d+(r-1)k).$$
We can replace $\alpha$ by $\alpha+m\Theta$ for $m\in\Z$ and assume that $d\geq0$; we further assume that $d\neq k$. This shows in particular that $(\alpha^r\cdot\Theta^{n-r})\neq 0$ for all $r$.

Assume that $n>2$ (the case $n=2$ is trivial). Let $p$ be a prime such that $p^s\mid n!$ with $s\in\Z_{>0}$ maximal and let $t$ be the largest integer such that $p^t\mid d-k$. We wish to prove that $t\geq s$. Assume the contrary; that is, assume that $t<s$. We then have that $p^{(s-t)(r-1)}\mid d+(r-1)k$ for every $r$. Then $p^{(s-t)(n-2)}\mid k$, and therefore $p^{(s-t)(n-2)}\mid d$. In particular, $p^{(s-t)(n-2)}\mid d-k$, and so $(s-t)(n-2)\leq t<s$. But then $s\geq n-1$, and so by Lemma~\ref{13} necessarily $p=2$ and $n$ is a power of $2$. This means that for every odd prime that divides $n!$, the same prime divides $d-k$ with the same or greater power. We have now reduced the proof to showing that the same is also true when $p=2$.

With $p=2$, we have that $s=n-1$, and so $(n-1-t)(n-2)<n-1$. After rearranging, we have that $t>n-1-\frac{n-1}{n-2}$, and so $t\geq n-1=s$.
\end{proof}

We can now state and prove our second main theorem.

\begin{theo}\label{16}
Let $(A,\Theta)$ be a ppav and let $d>0$. Then there is a bijective correspondence between abelian divisors on $A$ of degree $d$ and primitive numerical classes $[\alpha]\in \mbox{NS}(A,\Theta)$ that satisfy $q_r(\alpha)=(-1)^rd^r$ for $r=2,\ldots,n$, given by $Z\mapsto[Z]$.
\end{theo}
\begin{proof}
First we will prove injectivity. If $[Z]=[Y]$ in $\mbox{NS}(A,\Theta)$, then $Z\equiv Y+m\Theta$ for some $m\in\Z$.  By squaring and then intersecting with $\Theta^{n-2}$, we get that
\begin{eqnarray}
\nonumber m(2\deg Y+m(\Theta^n))&=&0\\
\nonumber m(-2\deg Z+m(\Theta^n))&=&0
\end{eqnarray}
But this is only possible if $m=0$, since $\deg Z$ and $\deg Y$ are positive.

For surjectivity, first let $[\alpha]\in \mbox{NS}(A,\Theta)$ be a primitive class such that $q_r(\alpha)=(-1)^rd^r$. By Lemma~\ref{15}, we get that $\deg\alpha-d\equiv0\mbox{ (mod }n!)$, and we define
$$\beta:=\alpha-\frac{\deg\alpha-d}{n!}[\Theta].$$
Since $\alpha$ is primitive in $\mbox{NS}(A,\Theta)$, it is trivial to see that $\beta$ is primitive in $\mbox{NS}(A)$. Moreover, $q_r(\beta)=q_r(\alpha)$ and $\deg\beta=d$. This means that $(\beta^r\cdot\Theta^{n-r})=0$ for all $2\leq r\leq n$, and so by Theorem~\ref{8}, $\beta$, and thus $\alpha$, comes from an abelian divisor.
\end{proof}

\begin{coro}
The map that takes an elliptic subgroup $E$ on $A$ to the numerical class of its abelian complement induces a bijection between elliptic subgroups of degree $d$ and primitive numerical classes $[\alpha]\in\mbox{NS}(A,\Theta)$ that satisfy $q_r(\alpha)=(-1)^r(n-1)!d^r$ for $r=2,\ldots,n$. 
\end{coro}

\begin{proof}
This follows from the previous theorem and from Lemma~\ref{12}. 
\end{proof}

It may seem that the forms $q_r$ for $r\geq3$ are extraneous, especially since Kani's characterization of elliptic curves on an abelian surface is by means of only one quadratic form. However, all the forms $q_r$ are needed for this characterization. For example, if $A=E_1\times E_2\times E_3$ for elliptic curves $E_i$ and $\Theta$ is the product polarization, let
$$D_1:=\{0\}\times E_2\times E_3$$
$$D_2:= E_1\times \{0\}\times E_3$$
$$D_3:= E_1\times E_2\times\{0\}.$$
For $k\in\Z_{>0}$, let $\alpha_k:=-kD_1+k(k+1)D_2+(k+1)D_3$. We have that $\deg\alpha_k=2(k^2+k+1)$, $(\alpha_k^2\cdot\Theta)=0$ and $(\alpha_k^3)=-k^2(k+1)^2$. We see that $\alpha_k$ is primitive and $q_2(\alpha_k)$ is a square, but $\alpha_k$ does not come from an abelian divisor (since $(\alpha_k^3)\neq0$). This shows that the form $q_3$ is indispensable here. Similar examples can be found in higher dimension.

In the case that $\Theta$ is not a principal polarization we cannot be as explicit as in Theorem~\ref{16}, but something can be said. One of the main obstructions to obtaining a similar theorem in the non-principally polarized case is the fact that abelian divisors are not necessarily primitive in $\mbox{NS}(A,\Theta)$. Nonetheless, we can still use the forms $q_r$ to find abelian divisors.

\begin{prop}
Let $(A,\Theta)$ be a polarized abelian variety. If $[\alpha]\in\mbox{NS}(A,\Theta)$ and $d\in\Z_{>0}$ such that $\deg\alpha\equiv d\mbox{ (mod }(\Theta^n))$ and $q_r(\alpha)=(-1)^rd^r$ for $r\leq n$, then $\alpha=m[Z]$ for some abelian divisor $Z$ on $A$ and some $m\in\Z$.
\end{prop}

\begin{proof}
Take $\beta=\alpha-\frac{\deg\alpha-d}{(\Theta^n)}[\Theta]\in\mbox{NS}(A)$. We see that $\deg\beta=d$ and $q_r(\beta)=(-1)^rd^r$ for $r\leq n$, and so $(\beta^r\cdot\Theta^{n-r})=0$ for $r\geq 2$. Since $\deg\beta>0$, Proposition~\ref{4} says that $\beta$ is effective and Proposition~\ref{3} says that $\beta$ is algebraically equivalent to a multiple of an abelian divisor.
\end{proof}

Let $x=(d_2,\ldots,d_n)\in\Z^{n-1}$ be a vector. We say that a polarized abelian variety $(A,\Theta)$ \emph{represents} $x$ if there exists a class $\alpha\in\mbox{NS}(A,\Theta)$) such that $q_r(\alpha)=d_r$. We say that it $\emph{primitively represents}$ the same vector if there is a primitive class that satisfies the same equation.

\begin{prop}\label{17}
Let $(A,\Theta)$ be a ppav of dimension $n$. Then $(A,\Theta)$ is isomorphic to a product abelian variety $(E\times Y,\mbox{pr}_1^*(0)+\mbox{pr}_2^*\Theta_2)$ for $E$ an elliptic curve and $(Y,\Theta_2)$ an $n-1$ dimensional polarized abelian variety if and only if $(A,\Theta)$ represents the vector $((-1)^r(n-1)!^r)_{r=2}^n$. This is equivalent to the existence of an elliptic curve $E$ in $A$ with $(\Theta\cdot E)=1$.
 \end{prop}
\begin{proof} If $A$ splits in the way that is stated, then $(\Theta\cdot (E\times\{0\}))=1$, and so Lemma~\ref{12} says that the abelian complement of $E\times\{0\}$ has degree $(n-1)!$.

For the other direction, we will first show that if $q_r(\alpha)$ satisfies the equation above, then $\alpha$ must be primitive. If $\alpha=m\beta$ for some primitive $\beta$ (we can assume $m$ positive), then $q_2(\beta)=((n-1)!/m)^2$, and so $m\mid(n-1)!$. But then $q_r(\beta)=(-1)^r((n-1)!/m)^r$, and so by Theorem~\ref{16} there exists an abelian divisor $Y$ on $A$ with $\deg Y=(n-1)!/m$. If $m>1$, this contradicts Lemma~\ref{12}.

Now assume that $q_r(\alpha)=((-1)^r((n-1)!)^r)_{r=2}^n$. Since $\alpha$ is primitive, there exists an abelian divisor $Z$ of degree $(n-1)!$ such that $[Z]=[\alpha]$ in $\mbox{NS}(A,\Theta)$. By Lemma~\ref{12}, the abelian complement of $Z$ is an elliptic curve $E$ with $(\Theta\cdot E)=1$. Now
$$((\Theta-Z)^r\cdot\Theta^{n-r})=(n-r)(n-1)!\geq0,$$
and by Proposition~\ref{4} $\Theta-Z\equiv D$ for some effective divisor $D$. This implies that $((\Theta-Z)\cdot E)\geq0$, and so $1-(Z\cdot E)\geq0$. But then $E$ intersects $Z$ in only one point, and so the addition map $E\times Z\to A$ is an isomorphism of varieties.

Since $(D\cdot E)=0$, if we see $D$ as a divisor on $E\times Z$, this means that $\mathcal{O}_{E\times Z}(D)|_{E\times\{z\}}$ is trivial for every $z\in Z$. By the Seesaw Theorem, we get that $\mathcal{O}_{E\times Z}(D)\simeq\mbox{pr}_2^*\mathcal{O}_{E\times Z}(\Theta_2)$ for some divisor $\Theta_2$ on $Z$. Summing everything up, we get that $\Theta$, seen as a divisor on the product, is numerically equivalent to $\mbox{pr}_1^*(0)+\mbox{pr}_2^*(\Theta_2)$.
\end{proof}

As a corollary, we obtain a nice geometric result in dimension 3.

\begin{coro}\label{18}
A ppav $(A,\Theta)\in\mathcal{A}_3$ is not the Jacobian of a curve if and only if $(q_2,q_3)$ represents $(4,-8)$.
\end{coro}

\begin{proof}
It is known that a principally polarized abelian 3-fold $(A,\Theta)$ is the Jacobian of some curve if and only if it is indecomposable. Therefore, by Proposition~\ref{17}, $(A,\Theta)$ is not the Jacobian of a curve if and only if it represents $(4,-8)$.
\end{proof}

Assume now that $(JC,\Theta_C)$ is the Jacobian of a curve $C$. A \emph{minimal elliptic cover} is a finite morphism $f:C\to E$ to an elliptic curve $E$ that does not factor through any other elliptic curve non-trivially. Two covers $f:C\to E$ and $f':C\to E'$ are \emph{isomorphic} if there is an isomorphism $\phi:E\to E'$ such that $\phi\circ f=f'$. Kani \cite{Kani} gives the following classification:

\begin{prop}\label{19}
The map $f\mapsto f^*E$ gives a 1-1 correspondence between the set of isomorphism classes of minimal elliptic covers $f:C\to E$ of degree $k$ and elliptic subgroups $E\leq JC$ with $(E\cdot\Theta_C)=k$.
\end{prop}

Translating this to our language, we get:

\begin{prop}\label{25}
Let $C$ be a curve of genus $g$, and let $JC$ be its Jacobian. Then there is a bijective correspondence between the following sets:
\begin{enumerate}
\item Isomorphism classes of minimal elliptic covers $C\to E$ of degree $k$.
\item Elliptic subgroups $E\leq JC$ such that $(E\cdot\Theta_C)=k$.
\item Primitive elements $\alpha\in\mbox{NS}(JC,\Theta_C)$ such that $q_r(\alpha)=(-1)^r(g-1)!^rk^r$ for $r=2,\ldots,g$.
\end{enumerate}
\end{prop}

\begin{coro}\label{20}
A 3-dimensional ppav $(A,\Theta)$ is the Jacobian of a genus 3 curve and splits isogenously as the product of elliptic curves if and only if $(q_2,q_3)$ does not represent $(4,-8)$ but there exist two distinct primitive elements in $\mbox{NS}(A,\Theta)$ that represent vectors of the form $(d^2,-d^3)$ for $d>2$.
\end{coro}

\section{Analytic calculations on $\mathcal{A}_n$}
Throughout this section, $(A,\Theta)$ will be a ppav over $\C$. Since we are working over the field of complex numbers, we can assume that $A=\C^n/\Lambda$, where $\Lambda$ is a rank $2n$ lattice in $\C^n$. Furthermore, we may assume that $\Lambda=(\tau\hspace{0.1cm}I)\Z^g$, where $\tau\in\frak{H}_n:=\{N\in M_n(\C):N=N^t,\mbox{Im}N>0\}$. We have a map
$$c_1:\mbox{Pic}(A)\simeq H^1(A,\mathcal{O}_A^\times)\to H^2(A,\Z)$$
that assigns to each $L\in\mbox{Pic}(A)$ its first Chern class $c_1(L)$. The first Chern class can be seen as a Hermitian form on $\C^n$ whose imaginary part takes on integer values on $\Lambda$. Actually, the image of $c_1$ is isomorphic to $\mbox{NS}(A)$, and so in this section we will write numerical classes as integral differential forms. It can be shown that $\mbox{NS}(A)=H^2(A,\Z)\cap H^{1,1}(A,\Z)$, and this means that an integral cohomology class $\omega$ is in $\mbox{NS}(A)$ if and only if $\omega\wedge dz_1\wedge\cdots\wedge dz_n=0$. Intersection of line bundles can be shown to be equal to
$$(L_1\cdots L_n)=\int_Ac_1(L_1)\wedge\cdots\wedge c_1(L_n).$$
Choose a symplectic basis $\{\lambda_1,\ldots,\lambda_{2n}\}$ for $\Lambda$, and take $x_i$ to be the real coordinate function of $\lambda_i$. By Lemma 3.6.4 of \cite{BL}, we have that
$$c_1(\Theta)=-\sum_{i=1}^ndx_i\wedge dx_{i+n}.$$
Given an integral form $\omega\in \mbox{NS}(A)$, we wish to find the intersection number $(\omega^r\cdot c_1(\Theta)^{n-r})$. Using the explicit basis we just wrote, it is easy to see that
$$c_1(\Theta)^{\wedge s}=(-1)^ss!\sum_{1\leq i_1<\cdots<i_s\leq n}dx_{i_1}\wedge dx_{i_1+n}\wedge\cdots\wedge dx_{i_s}\wedge dx_{i_s+n}.$$
In particular, for $s=n$, we get that
$$n!=(c_1(\Theta)^n)=\int_Ac_1(\Theta)^{\wedge n}=(-1)^nn!\int_Adx_{1}\wedge dx_{n+1}\wedge\cdots\wedge dx_{n}\wedge dx_{2n}.$$
If we put $\eta:=dx_{1}\wedge dx_{n+1}\wedge\cdots\wedge dx_{n}\wedge dx_{2n}$, this implies that when calculating intersection numbers, we have
$$(L_1\cdots L_n)=(-1)^n\cdot\mbox{coefficient of }\eta\mbox{ in }c_1(L_1)\wedge\cdots\wedge c_1(L_n).$$
From now on, let 
$$\omega=\sum_{i<j}a_{ij}dx_i\wedge dx_j\in\mbox{NS}(A)$$
be a differential form.

\begin{lemma}\label{22}
The degree of $\omega$ is $-(n-1)!(a_{1,n+1}+a_{2,n+2}+\cdots+a_{n,2n})$.
\end{lemma}
\begin{proof}
This is easily written out by hand. 
\end{proof}

It is clear that the formulas for the $q_r$ for a general $n$ are complicated to write out by hand. For small $n$, however, we can do this. For example, for $n=3$, we can write the forms $q_r$ as follows:
\begin{eqnarray}
\nonumber q_2(\omega)&=&12a_{12}a_{45}+12a_{13}a_{46}+4a_{14}^2-4a_{14}a_{25}-4a_{14}a_{36}+12a_{15}a_{24}+\\
\nonumber && 12a_{16}a_{34}+12a_{23}a_{56}+4a_{25}^2-4a_{25}a_{36}+12a_{26}a_{35}+4a_{36}^2\\
\nonumber &&\\
\nonumber q_3(\omega)&=&36a_{12}a_{14}a_{45}+36a_{12}a_{25}a_{45}-108a_{12}a_{34}a_{56}+
108a_{12}a_{35}a_{46}-\\
\nonumber &&72a_{12}a_{36}a_{45}+36a_{13}a_{14}a_{46}+108a_{13}a_{24}a_{56}-72a_{13}a_{25}a_{46}+\\
\nonumber && 108a_{13}a_{26}a_{45}+36a_{13}a_{36}a_{46}+8a_{14}^3-12a_{14}^2a_{25}-\\
\nonumber && 12a_{14}^2a_{36}+36a_{14}a_{15}a_{24}+36a_{14}a_{16}a_{34}-72a_{14}a_{24}a_{56}-\\
\nonumber &&12a_{14}a_{25}^2+48a_{14}a_{25}a_{36}-72a_{14}a_{26}a_{35}-12a_{14}a_{36}^2+\\
\nonumber && 108a_{15}a_{23}a_{46}+36a_{15}a_{24}a_{25}-72a_{15}a_{24}a_{36}+108a_{15}a_{26}a_{34}-\\
\nonumber && 108a_{16}a_{23}a_{45}+108a_{16}a_{24}a_{35}-72a_{16}a_{25}a_{34}+36a_{16}a_{34}a_{36}+\\
\nonumber && 36a_{23}a_{25}a_{56}+36a_{23}a_{36}a_{56}+8a_{25}^3-12a_{25}^2a_{36}+36a_{25}a_{26}a_{35}-\\
\nonumber && 12a_{25}a_{36}^2+36a_{26}a_{35}a_{36}+8a_{36}^3.
\end{eqnarray}

Let us keep assuming that $n=3$. Since $H^2(A,\Z)\simeq\bigwedge^2H^1(A,\Z)\simeq\bigwedge^2\mbox{Hom}(\Lambda,\Z)\simeq\Z^{15}$ and
$$c_1(\Theta)=-(dx_1\wedge dx_4+dx_2\wedge dx_5+dx_3\wedge dx_6),$$
we can take the projection
$$\Z^{15}\to H^2(A,\Z)/\Z c_1(\Theta)\simeq\Z^{14}$$
that takes
$$(a_{ij})_{1\leq i<j\leq6}\mapsto(a_{ij}')_{1\leq i<j\leq6,(i,j)\neq (3,6)},$$
where $a'_{ij}=a_{ij}$ if $(i,j)\neq(1,4),(2,5)$, and $a_{14}'=a_{14}-a_{36}$ and $a_{25}'=a_{25}-a_{36}$.
From what we said above, we have that an integral form $\omega$ is in $\mbox{NS}(A)$ if and only if $\omega\wedge dz_1\wedge dz_2\wedge dz_3=0$. Using the change of coordinates from real coordinates to complex coordinates by the matrix $(\tau\hspace{0.1cm}I)$, we get the following proposition:

\begin{prop}\label{23}
Let $(A,\Theta)$ be a ppav of dimension 3, corresponding to a matrix $\tau=(\tau_{ij})\in\frak{H}_3$. Then $\mbox{NS}(A,\Theta)$ consists of all vectors $\beta=(b_1,\ldots,b_{14})\in\Z^{14}$ that satisfy the linear equations
\begin{eqnarray}
\nonumber0&=&b_6-\tau_{13}b_7-\tau_{23}b_{8}-\tau_{33}b_9+\tau_{12}b_{10}+\tau_{22}b_{11}+
\left|\begin{array}{cc}\tau_{12}&\tau_{13}\\\tau_{22}&\tau_{23}\end{array}\right|b_{12}+\\
\nonumber&&\left|\begin{array}{cc}\tau_{12}&\tau_{13}\\\tau_{23}&\tau_{33}\end{array}\right|b_{13}+
\left|\begin{array}{cc}\tau_{22}&\tau_{23}\\\tau_{23}&\tau_{33}\end{array}\right|b_{14}\\
\nonumber0&=&b_2-\tau_{13}b_3-\tau_{23}b_4-\tau_{33}b_5+\tau_{11}b_{10}+\tau_{12}b_{11}+
\left|\begin{array}{cc}\tau_{11}&\tau_{12}\\\tau_{13}&\tau_{23}\end{array}\right|b_{12}+\\
\nonumber&&\left|\begin{array}{cc}\tau_{11}&\tau_{13}\\\tau_{13}&\tau_{33}\end{array}\right|b_{13}+
\left|\begin{array}{cc}\tau_{12}&\tau_{13}\\\tau_{23}&\tau_{33}\end{array}\right|b_{14}\\
\nonumber0&=&b_1-\tau_{12}b_3-\tau_{22}b_4-\tau_{23}b_5+\tau_{11}b_7+\tau_{12}b_8+\tau_{13}b_9+
\left|\begin{array}{cc}\tau_{11}&\tau_{12}\\\tau_{12}&\tau_{22}\end{array}\right|b_{12}+\\
\nonumber&&\left|\begin{array}{cc}\tau_{11}&\tau_{12}\\\tau_{13}&\tau_{23}\end{array}\right|b_{13}+
\left|\begin{array}{cc}\tau_{12}&\tau_{13}\\\tau_{22}&\tau_{23}\end{array}\right|b_{14}\\
\nonumber0&=&\tau_{13}b_1-\tau_{12}b_2+
\left|\begin{array}{cc}\tau_{12}&\tau_{13}\\\tau_{22}&\tau_{23}\end{array}\right|b_4+
\left|\begin{array}{cc}\tau_{12}&\tau_{13}\\\tau_{23}&\tau_{33}\end{array}\right|b_5+\tau_{11}b_6-
\left|\begin{array}{cc}\tau_{11}&\tau_{12}\\\tau_{13}&\tau_{23}\end{array}\right|b_8-\\
\nonumber&&\left|\begin{array}{cc}\tau_{11}&\tau_{13}\\\tau_{13}&\tau_{33}\end{array}\right|b_9+
\left|\begin{array}{cc}\tau_{11}&\tau_{12}\\\tau_{12}&\tau_{22}\end{array}\right|b_{11}+(\det\tau)b_{14}\\
\nonumber0&=&-\tau_{23}b_1+\tau_{22}b_2+\left|\begin{array}{cc}\tau_{12}&\tau_{13}\\\tau_{22}&\tau_{23}\end{array}\right|b_3-
\left|\begin{array}{cc}\tau_{22}&\tau_{23}\\\tau_{23}&\tau_{33}\end{array}\right|b_{5}-\tau_{12}b_6-
\left|\begin{array}{cc}\tau_{11}&\tau_{12}\\\tau_{13}&\tau_{23}\end{array}\right|b_7+\\
\nonumber&&\left|\begin{array}{cc}\tau_{12}&\tau_{13}\\\tau_{23}&\tau_{33}\end{array}\right|b_9+
\left|\begin{array}{cc}\tau_{11}&\tau_{12}\\\tau_{12}&\tau_{22}\end{array}\right|b_{10}+(\det\tau)b_{13}\\
\nonumber0&=&\tau_{33}b_1-\tau_{23}b_2-\left|\begin{array}{cc}\tau_{12}&\tau_{13}\\\tau_{23}&\tau_{33}\end{array}\right|b_3
-\left|\begin{array}{cc}\tau_{22}&\tau_{23}\\\tau_{23}&\tau_{33}\end{array}\right|b_4+\tau_{13}b_6+
\left|\begin{array}{cc}\tau_{11}&\tau_{13}\\\tau_{13}&\tau_{33}\end{array}\right|b_7+\\
\nonumber&&\left|\begin{array}{cc}\tau_{12}&\tau_{13}\\\tau_{23}&\tau_{33}\end{array}\right|b_8-
\left|\begin{array}{cc}\tau_{11}&\tau_{12}\\\tau_{13}&\tau_{23}\end{array}\right|b_{10}-
\left|\begin{array}{cc}\tau_{12}&\tau_{13}\\\tau_{22}&\tau_{23}\end{array}\right|b_{11}+(\det\tau)b_{12}
\end{eqnarray}
\end{prop}

\vspace{0.5cm}

After rewriting the equations for $q_r$ in $\Z^{14}$ and using Theorem~\ref{16}, we get the following corollary:

\begin{coro}\label{24}
The ppav of dimension 3 corresponding to a matrix $\tau=(\tau_{ij})\in\frak{H}_3$ contains an elliptic curve whose abelian complement has degree $d$ if and only if there exists a non-zero vector $\beta=(b_1,\ldots,b_{14})\in\Z^{14}$ that satisfies the equations of Proposition~\ref{23} and such that
\begin{eqnarray}\nonumber d^2&=&12b_1b_{12}+12b_2b_{13}+4b_8^2-4b_3b_8+4b_3^2+12b_4b_7+12b_5b_{10}+12b_6b_{14}+12b_9b_{11}\\
\nonumber-d^3&=&108b_4b_9b_{10}+36b_1b_8b_{12}-72b_5b_8b_{10}+36b_3b_5b_{10}-108b_5b_6b_{12}+\\
\nonumber&&108b_1b_{11}b_{13}-72b_3b_7b_{14}+36b_9b_{11}b_8+108b_2b_7b_{14}-72b_2b_8b_{13}-\\
\nonumber&&72b_3b_9b_{11}+108b_4b_6b_{13}+36b_3b_4b_7+36b_4b_7b_8+36b_2b_3b_{13}+\\
\nonumber&&108b_5b_7b_{11}-108b_1b_{10}b_{14}-12b_3b_8^2+108b_2b_9b_{12}+36b_6b_8b_{14}+\\
\nonumber&&36b_1b_3b_{12}+8b_3^3+8b_8^3.
\end{eqnarray}

\end{coro}

In the equations above, there is nothing special about the particular case $n=3$, except for the fact that the equations are not extremely long. With higher dimensions, similar equations can be found using the same method. In fact, using \cite{vdG} chapter IX as inspiration, we could define the variety
$$\mathcal{A}_{g,d}:=\{(A,\Theta): A\mbox{ contains an elliptic curve of degree }d\};$$
these varieties seem to be the correct generalization of Humbert surfaces. It is not clear at first that the above set is a variety (as it could have infinitely many components), but work done by Debarre in \cite{Debarre} shows that it is actually irreducible, being a cover of $\mathbb{H}\times\frak{H}_{g-1}$. These varieties seem to be of interest in their own right.

We finish with an example that shows how this theory can be used to concretely find elliptic curves.

\begin{ex}
Gonz\'alez-Aguilera and Rodr\'{\i}guez \cite{GR} found, for every $n\geq 3$, a family of indecomposable principally polarized abelian varieties each of whose underlying abelian variety is isomorphic to the product of elliptic curves. More specifically, they give the family
$$\mathcal{F}_n:=\{\sigma\tau_0:\sigma\in\mathbb{H}\}\subseteq\frak{H}_n,$$
where
$$\tau_0=\left(\begin{array}{cccc}n&-1&\cdots&-1\\-1&n&\cdots&-1\\
\vdots&\vdots&\ddots&\vdots\\-1&-1&\cdots&n\end{array}\right).$$
If $(A_\sigma,\Theta_\sigma)$ denotes the ppav associated to $\sigma \tau_0$, then
$$A_\sigma\simeq E_{(n+1)\sigma}^{n-1}\times E_{\sigma},$$
where $E_\sigma:=\C/\langle1,\sigma\rangle$. For $n=3$, the equations of Proposition~\ref{23} become
\begin{eqnarray}
\nonumber 0&=&b_6+\sigma b_7+\sigma b_8-3\sigma b_9-\sigma b_{10}+3\sigma b_{11}+4\sigma^2 b_{12}-4\sigma^2 b_{13}+8\sigma^2 b_{14}\\
\nonumber 0&=&b_2+\sigma b_3+\sigma b_4-3\sigma b_5+3\sigma b_{10}-\sigma b_{11}-4\sigma^2 b_{12}+8\sigma^2 b_{13}-4\sigma^2 b_{14}\\
\nonumber 0&=&b_1+\sigma b_3-3\sigma b_4+\sigma b_5+3\sigma b_7-\sigma b_8-\sigma b_9+8\sigma^2 b_{12}-4\sigma^2 b_{13}+4\sigma^2 b_{14}\\
\nonumber 0&=&-\sigma b_1+\sigma b_2+4\sigma^2b_4-4\sigma^2 b_5+3\sigma b_6+4\sigma^2 b_8-8\sigma^2b_9+8\sigma^2b_{11}+16\sigma^3b_{14}\\
\nonumber 0&=&\sigma b_1+3\sigma b_2+4\sigma^2b_3-8\sigma^2 b_5+\sigma b_6+4\sigma^2b_7-4\sigma^2b_9+8\sigma^2b_{10}+16\sigma^3b_{13}\\
\nonumber0&=&3\sigma b_1+\sigma b_2+4\sigma^2b_3-8\sigma^2b_4-\sigma b_6+8\sigma^2 b_7-4\sigma^2b_8+4\sigma^2b_{10}-4\sigma^2b_{11}+16\sigma^3b_{12}
\end{eqnarray}

If $[\Q(\sigma):\Q]>2$, then if we have an integral solution to the equations above and $x_1,\ldots,x_{11}$ are the coefficients of the generators, we get the following relations:
\begin{eqnarray}
\nonumber0&=&-2x_1-x_6+x_{11}\\
\nonumber 0&=&x_1+x_6-2x_{11}\\
\nonumber0&=&-x_1-2x_6+x_{11}\\
\nonumber 0&=&-x_3+3x_4-x_5-3x_8+x_9+x_{10}\\
\nonumber 0&=&x_2+3x_3-x_4-x_5-3x_7\\
\nonumber 0&=&-3x_2+x_7-x_8-x_9+3x_{10}
\end{eqnarray}
This necessarily leads to $x_1=x_6=x_{11}=0$. The general integral solution to this system of equations is
$$\begin{pmatrix}x_2\\x_3\\x_4\\x_5\\x_7\\x_8\\x_9\\x_{10}\end{pmatrix}=\begin{pmatrix}a+c-2d+e\\
3a+b+c+d\\a+c+d\\e\\3a+b+c\\c\\b\\c-2d+e
\end{pmatrix}$$
for $a,b,c,d,e\in\Z$. Therefore, a vector in $\Z^{14}$ is a solution for the N\'eron-Severi equations if and only if it is of the form
$$\begin{pmatrix}0\\0\\e\\a+c+d\\3a+b+c+d\\0\\c\\b\\c-2d+e\\3a+b+c\\a+c-2d+e\\0\\0\\0\end{pmatrix}
=a\begin{pmatrix}0\\0\\0\\1\\3\\0\\0\\0\\0\\3\\1\\0\\0\\0\end{pmatrix}+
b\begin{pmatrix}0\\0\\0\\0\\1\\0\\0\\1\\0\\1\\0\\0\\0\\0\end{pmatrix}+
c\begin{pmatrix}0\\0\\0\\1\\1\\0\\1\\0\\1\\1\\1\\0\\0\\0\end{pmatrix}+
d\begin{pmatrix}0\\0\\0\\1\\1\\0\\0\\0\\-2\\0\\-2\\0\\0\\0\end{pmatrix}+
e\begin{pmatrix}0\\0\\1\\0\\0\\0\\0\\0\\1\\0\\1\\0\\0\\0\end{pmatrix}.$$
Using this basis, we will write an element of $\mbox{NS}(A_\sigma,\Theta_\sigma)$ as a 5-tuple $(a,b,c,d,e)$. The forms can then be written as
\begin{eqnarray}\nonumber q_2(a,b,c,d,e)&=&108a^2+16b^2+36c^2+48d^2+16e^2+72ab+96ac+12ad+\\
\nonumber &&12ae+24bc+12bd-4be-24cd+24ce-48de\\
\nonumber q_3(a,b,c,d,e)&=&576ead-504bcd+936cea+24e^2b+24eb^2-504bad-\\
\nonumber&&1080cad+576ced+360bae+360bce-144bca-64e^3-64b^3+\\
\nonumber&&216bc^2-72bd^2+216c^2e-144ce^2-288ed^2+288de^2-72ae^2+\\
\nonumber&&216c^3+648ea^2+648ca^2+864c^2a-216c^2d-432cd^2-\\
\nonumber&&648a^2d-648d^2a-432ab^2-144b^2c-72db^2-648ba^2\end{eqnarray}

Using a simple Java program, we can find many primitive elements $\alpha$ such that $q_2(\alpha)$ is a square $d^2$ and such that $q_3(\alpha)=-d^3$. For example, the following table shows all abelian divisors in $\mbox{NS}(A_\sigma,\Theta_\sigma)$ whose coordinates lie between $-3$ and $3$ and whose degree is less than or equal to $6$:\\

\begin{center}
\begin{tabular}{|l| c| c| c|}
\hline
Divisor class of $Z$ & $(Z\cdot\Theta_\sigma^2)$ & $(E\cdot\Theta_\sigma)$\\
\hline
(0,0,0,-1,-2)&4&2\\
(1,-1,-1,-1,-1)&4&2\\
(0,-1,1,0,-1)&4&2\\
(1,-2,-1,0,0)&4&2\\
(0,1,0,0,0)&4&2\\
(-1,1,1,0,0)&4&2\\
(0,0,0,0,1)&4&2\\
(0,0,0,1,1)&4&2\\
(-1,2,0,1,2)&4&2\\
(0,0,1,-1,-3)&6&3\\
(0,0,-1,0,0)&6&3\\
(1,-3,0,0,0)&6&3\\
(-1,3,0,1,3)&6&3\\
\hline
\end{tabular}
\end{center}

Although this example was studied in \cite{GR} by using automorphisms of the elements of $\mathcal{F}_n$, we emphasize that the above elliptic curves were found by only using the period matrix of each element and we did not rely on the existence of automorphisms.

\end{ex}

\bibliographystyle{alpha}

\end{document}